\documentclass[12pt,leqno]{amsart}
\usepackage{verbatim}
\usepackage{amsthm}
\usepackage{amsmath}
\usepackage{amssymb}
\usepackage{mathrsfs}
\usepackage{todonotes}

\usepackage{color}

\begin{document}

\numberwithin{equation}{section}
\renewcommand{\theequation}{\thesection.\arabic{equation}}
\reversemarginpar

\newcommand{\Ac}{\mathcal A}
\newcommand{\Asc}{\mathscr A}
\newcommand{\Bc}{\mathcal B}
\newcommand{\Bsc}{\mathscr B}
\newcommand{\Ec}{\mathcal E}
\newcommand{\Ectau}{\mathcal{E}(\tau)}
\newcommand{\Ect}{\tilde{\Ec}}
\newcommand{\Esc}{\mathscr E}
\newcommand{\eps}{\epsilon}
\newcommand{\Fc}{\mathcal F}
\newcommand{\Fsc}{\mathscr{F}}
\newcommand{\Fsco}{\mathscr{F}_0}
\newcommand{\Gamk}{\Gamma_k}
\newcommand{\Gsc}{\mathscr{G}}
\newcommand{\Ic}{\mathcal{I}}
\newcommand{\Isc}{\mathscr{I}}
\newcommand{\Lsc}{\mathscr{L}}
\newcommand{\Mcl}{{\mathcal M} \, }
\newcommand{\Msc}{\mathscr M}
\newcommand{\Msctau}{\mathscr{M}_{\tau}}
\newcommand{\N}{\mathbb{N}}
\newcommand{\Nb}{\mbox{$\mathbb N$}}
\newcommand{\Om}{\Omega}
\newcommand{\bdy}{\partial\Omega}
\newcommand{\Omb}{\overline{\Omega}}
\newcommand{\pal}{\partial}
\newcommand{\Pc}{\mathcal P}
\newcommand{\Qc}{\mathcal Q} 
\newcommand{\R}{\mathbb R}
\newcommand{\Rc}{\mathcal R}
\newcommand{\Rb}{\overline{\R}}
\newcommand{\Rn}{{\R}^n}
\newcommand{\RN}{{\R}^N}
\newcommand{\Rp}{[0,\infty)}
\newcommand{\ra}{\rightarrow}
\newcommand{\Sc}{\mathcal S}
\newcommand{\ub}{\bar{u}}
\newcommand{\uh}{\hat{u}}
\newcommand{\ut}{\tilde{u}}
\newcommand{\vt}{\tilde{v}}
\newcommand{\ds}{\displaystyle}
\newcommand{\ve}{\varepsilon}
\newcommand{\lam}{\lambda}


\newcommand{\aip}[1]{\langle#1\rangle}
\newcommand{\mtau}{m_{\tau}}
\newcommand{\n}[1]{\left\vert#1\right\vert}
\newcommand{\nm}[1]{\left\Vert#1\right\Vert}


\newcommand{\Div}{\mathop{\rm div}\nolimits}
\newcommand{\dsg}{d \sigma}
\newcommand{\gradu}{\nabla u}
\newcommand{\gradv }{\nabla v}
\newcommand{\Gscmu}{\Gsc(.,\mu)}
\newcommand{\Ibdy}{ \int_{\bdy}} 
\newcommand{\IOm}{\int_{\Om}}

\newcommand{\tlam}{\tilde{\lambda}}
\newcommand{\lamo}{\lambda_1}
\newcommand{\lamt}{\lambda_2}
\newcommand{\lami}{\lambda_i}
\newcommand{\lamj}{\lambda_j}
\newcommand{\lamk}{\lambda_k}
\newcommand{\laml}{\lambda_l}
\newcommand{\lamotau}{\lambda_1(\tau)}
\newcommand{\lamttau}{\lambda_2(\tau)}
\newcommand{\lamjtau}{\lambda_j(\tau)}

\newcommand{\Lamtau}{\Lambda(\tau)}

\newcommand{\eotau}{e_1(\tau)}
\newcommand{\ejtau}{e_j(\tau)}


\newcommand{\barr}{\begin{eqnarray}}
\newcommand{\beq}{\begin{equation}}
\newcommand{\bpf}{\begin{proof} \quad}
\newcommand{\btm}{\begin{thm}}
\newcommand{\blem}{\begin{lem}}
\newcommand{\elem}{\end{lem}}
\newcommand{\bcor}{\begin{cor}}
\newcommand{\ecor}{\end{cor}}
\newcommand{\earr}{\end{eqnarray}}
\newcommand{\eeq}{\end{equation}}
\newcommand{\epf}{\end{proof}}
\newcommand{\etm}{\end{thm}}
\newcommand{\beqq}{\begin{equation*}}
\newcommand{\eeqq}{\end{equation*}}


\newcommand{\Elam}{E_{\lambda}}
\newcommand{\Hone}{H^1(\Om)}
\newcommand{\Hones}{\Hone \times \Hone} 
\newcommand{\Hz}{H_{0}^1(\Om)}
\newcommand{\HLap}{H(\Delta, \Omega)}
\newcommand{\Harm}{\mathcal{H}(\Omega)}
\newcommand{\HLapz}{H_0(\Delta, \Omega)}
\newcommand{\HLapzz}{H_{00}(\Delta, \Omega)}
\newcommand{\Linb}{L^{\infty}(\bdy,\dsg)}
\newcommand{\Lp}{L^p (\Om)}
\newcommand{\Lpb}{L^p (\bdy, \dsg)}
\newcommand{\Lq}{L^q (\Om)}
\newcommand{\Lt}{L^2 (\Om)}
\newcommand{\Ltb}{L^2 (\bdy,\dsg)}
\newcommand{\RxHone}{\R \times \Hone}

%
\newtheorem{thm}{Theorem}[section]
\newtheorem{cor}[thm]{Corollary}
\newtheorem{cond}{Condition}
\newtheorem{lem}[thm]{Lemma}
\theoremstyle{definition}
\newtheorem{rem}[thm]{Remark}

\title[A Nonlinear Fredholm Alternative]
{A Fredholm Alternative for Elliptic Equations with Interior and Boundary Nonlinear Reactions}
\author[Maroncelli \, \&  \, Rivas]{Dan Maroncelli \,  \&  \, M.A. Rivas}

\address{Daniel Maroncelli \newline
Department of Mathematics, College of Charleston \newline
66 George Street, Charleston, SC 29424 USA}
\email{maroncellidm@cofc.edu}

\address{Mauricio A. Rivas \newline
Department of Mathematics and Statistics, \, North Carolina A\&T State University \newline
1601 E. Market Street
Greensboro, NC 27411 USA}
\email{marivas@ncat.edu}

\date{\today}

\begin{abstract} 
In this paper we study the existence of solutions to the following generalized nonlinear two-parameter problem
\beqq
a(u, v) \; =\; \lambda\, b(u, m) + \mu\, m(u, v) + \varepsilon\, F(u, v),
\eeqq
for a triple $(a, b, m)$ of continuous, symmetric bilinear forms on a real separable Hilbert space $V$ and nonlinear form $F$.
This problem is a natural abstraction of nonlinear problems that occur for 
a large class of differential operators, various elliptic pde's with nonlinearities in either the differential equation and/or the boundary conditions being a special subclass. 
First, a Fredholm alternative for the associated linear two-parameter eigenvalue problem is developed, and then this is used to construct a nonlinear version of the Fredholm alternative.
Lastly, the Steklov-Robin Fredholm equation is used to exemplify the abstract results.
\end{abstract}
\maketitle

**\textbf{ This work has been accepted for publication in the journal- Topological Methods in Nonlinear Analysis.}

\section{\bf Introduction}



This paper studies the existence of solutions to the generalized nonlinear two-parameter problem given by
\beq\label{E.Non1}
a(u, v) \; =\; \lambda\, b(u, m) + \mu\, m(u, v) + \varepsilon\, F(u, v),
\eeq
where $(a, b, m)$ is a triple of continous, symmetric bilinear forms on a separable Hilbert space $V$ and $F$ is a nonlinear form.
Our work has two distinct parts. 
Formulated first is a {\it two-parameter Fredholm alternative} for the triple $(a, b, m)$.
The alternative and the explicit spectral formulae for the solutions are given in terms of the eigendata associated with the {\it variational eigencurves} of $(a, b, m)$ as obtained in \cite{RR19}. 
Then existence results to \eqref{E.Non1} for cases of {\it nonresonance} and {\it resonance} are described in detail.
The formulation \eqref{E.Non1} is called the nonlinear {\it Fredholm equation} for $(a, b, m)$ as it subsumes various versions of such equations.

Our abstract setting enables the study of large classes of PDEs posed in weak form since the conditions imposed on $a, b, m$ are simple enough that they allow treating, for example, very general second-order elliptic equations with (possibly) nonlinear boundary conditions posed in weak form.
This is because data coefficients of the equations may be quite general, including sign-changing weights, and  yet it is easy to verify that the bilinear forms arising from the weak formulation of these differential equations satisfy our general assumptions.
Moreover, the use of bilinear forms directly, as opposed to densely defined linear operators, facilitates imposing other boundary conditions, such as Robin, Steklov, or in some cases mixed boundary conditions, on these equations (see \S 8, 9, 10 of \cite{AR16} concerning tensor product spaces.).
The methodology provides for a straightforward, but not trivial, inspection of such problems and for much simpler proofs.

Our study is motivated by the renewed interest and applications of two-parameter nonlinear problems with non-Dirichlet boundary conditions.
For instance, in some recent papers of Shivaji and collaborators, such as \cite{ShivSon19}, \cite{ShivSon21}, or \cite{ShivRob18},  classes of singular or exact reaction diffusion equations where a parameter in the differential equation influences another parameter in the boundary condition are examined.
A related first eigencurve is often a tool in describing stability or bifurcation results for such problems.
Sign-changing eigenvalue problems for  triples of bilinear forms are studied by Kielty in \cite{K21} that are  applicable to, for example, nonlinear population dynyamics.
There the author describes how the spectrum of the problem converges as the negative part of the weight is taken to negative infinity.
This can be rephrased as increasing one eigenparameter while maintaining the other fixed.
See \cite{RR19} that cites other recent uses of eigencurves in advanced scientific applications.

An interesting related work is that of Peitgen and Schmitt in \cite{PS84}.
There the authors treat very general two-parameter second-order elliptic equations, but with the nonlinearity only on the differential equation and only the Dirichlet boundary condition being imposed.
Their method of proof includes linearization and (Learay-Schauder) continuation techniques.
A wealth of detailed examples, as well as numerical work, is also provided in \cite{PS84}.
The present work can be regarded as a supplement of the analysis of Peitgen and Schmitt as our approach is both topological and variational in nature.

In addition, the results obtained in this paper complement both the {geometrical results} in \cite{RR19} for the {variational eigencurves} of $(a, b, m)$ and the {solvability results} of Mavinga and Nakashama in \cite{MN10}.
In \cite{RR19}, attention is given to producing and describing the spectrum of the triple $(a, b, m)$.
Here, the focus is on using the spectrum of $(a, b, m)$ in studying the solvability of nonlinear elliptic problems with nonlinear boundary conditions.
In \cite{MN10}, existence results are obtained according to how their nonlinearities interact with so-called ``eigenvalue-lines'', and here we give existence results based on how the nonlinearity $F$ in \eqref{E.Non1} interacts with the variational eigencurves.

Similar spectral representations and decompositions as used here have recently been invoked by Manki Cho in \cite{Cho20} and  De Gournay, Allaire and Jouve in \cite{DeG08}
to provide numerical results in different areas.
In \cite{Cho20}, new {\it meshless} numerical methods are introduced in the study of regularized Laplacian boundary value problems that yield effective and accurate approximations.
In \cite{DeG08}, both theorerical and numerical results are given for the robust compliance problem that appears in shape and topology optimization.
A nice version of the Fredholm alternative is given in Section 2.3.1 of \cite{DeG08} and may be compared to that given here.

The paper is organized as follows.  In \S \ref{s.definitions}, the notation and assumptions are presented.
The {\it canonical eigendata}, and their related properties, for the triple $(a, b,m)$ of bilinear forms are detailed in \S \ref{s.canonicalData}. 
The two-parameter Fredholm Alternative for $(a, b, m)$ is stated and proved in \S \ref{s.LinearFredholm}.
The nonlinear Fredholm Alternative for $(a, b, m)$ is treated next, with the nonresonance case in \S \ref{s.Nonresonance} and the resonance case in \S \ref{s.Resonance}.
The abstract results are then exemplified in \S \ref{s.SRFE} that treats the {\it Steklov-Robin Fredholm Equation}.
To finish, the appendix in \S \ref{s.appendix} provides conditions for which the  Fr\'{e}chet differentiability requirement of our main results, theorems \ref{T:Nonres} and \ref{T:Res}, is satisfied. For simplicity, the result is stated for a class of nonlinearities in the one-dimensional setting, but higher-dimensional analogs exist under appropriate assumptions.

\vspace{1em}
\section{\bf Definitions and Notation}\label{s.definitions}

In this paper, $V$ will denote a real separable Hilbert space with inner product and norm given by $\langle \cdot, \cdot\rangle_V$ and $\|\cdot\|_V$, respectively, and whose dual space is $V^*$.
The {\it two-parameter eigenproblem} for the triple $(a, b, m)$ of bilinear forms on $V$ is that of finding $(\lambda, \mu)\in \R^2$ and nontrivial $e\in V$ satisfying
\beq\label{e.2Pproblem}
a(e, v) \; =\; \lambda\, b(e, v)  +  \mu\, m(e, v) \qquad \text{for all }v\in V.
\eeq
A nonzero $e$ is called an {\it eigenvector} associated to the {\it eigenpair} $(\lambda, \mu)$ of $(a, b, m)$, and the problem corresponding to \eqref{e.2Pproblem} will be called the $(a, b,m)$-{\it eigenproblem}.
The subset $\sigma(a, b, m)\subset \R^2$ comprising all eigenpairs $(\lambda, \mu)$ will be called the {\it spectrum} of $(a, b, m)$.

We will denote the associated quadratic forms on $V$ for $a,b$, and $m$ by $\mathcal{A}, \mathcal{B}$, and $\mathcal{M}$, respectively, so then 
\beq
\Ac(u) \; := \; a(u, u) \; , \qquad \Bc(u) \; := \; b(u, u) \; , \qquad \text{and}\qquad \mathcal{M}(u) \; : = \; m(u, u) \qquad \text{for all }u\in V.
\eeq
Our assumptions on the bilinear forms $a, b, m$ are:
\begin{itemize}
\item[\bf (A1):]   $a(\cdot, \cdot)$ is a coercive, continuous, symmetric bilinear form on $V$, so there are constants $0<k_0\leq k_1  < \infty$ such that
\beq
 \qquad k_0 \|u\|_V^2 \; \leq \; \mathcal{A}(u) \;  \leq \; k_1\|u\|_V^2  \qquad \text{for all }u\in V.
\eeq 
\item[\bf (A2):]  $b(\cdot, \cdot)$ is a weakly continuous, symmetric bilinear form on $V$, and $\mathcal{B}$ may be an indefinite form so that it could attain positive, negative, and zero values.
\item[\bf (A3):]   $m(\cdot, \cdot)$ is a weakly continuous, symmetric bilinear form on $V$ that satisfies
\beq
\mathcal{M}(u)\; > \; 0 \qquad \text{for all nonzero }u\in V.
\eeq
\end{itemize}

When these assumptions hold, both $a, m$ will be inner products on $V$, with $a(\cdot, \cdot)$ equivalent to the $V$-inner product, and there is an implicit compactness that underpins the spectral analysis for bilinear forms.

For concreteness, we mention here that assumptions A1-A3 do not hold for all weak formulations of differential equations with arbitrary weight functions. An easy example not satisfying our assumptions, often overlooked by experts, is
\[
V\; = \; H^1(\R), \quad a(u, v) \; : = \; \int_\R [u'v' +uv]\, dx , \quad b(u, v)\; := \; \int_\R b_0 \, u \, v\, dx \, \quad m(u, v)\; := \; \int_\R m_0\, u\, v\, dx,
\]
with $b_0, m_0\in L^\infty(\R)$ such that $m_0(x) >0$ for all $x\in \R$, since $m$ is not weakly continuous.
Thus, to apply the results given in this paper it is vital to verify that all assumptions on $a, b, m$ hold.

\vspace{1em}
\section{\bf Canonical spectral data for two-parameter eigenproblems}\label{s.canonicalData}

To determine the spectral data for the two-parameter problem \eqref{e.2Pproblem}, the problem in \cite{RR19} is shifted to get, for fixed $\lambda$, the one-parameter eigenproblem of finding $\mu+\tau$ and nontrivial $e$ satisfying
\beq\label{e.Pmutau}
a_{\lambda, \tau}(e, v)  \; =\; (\mu + \tau)\, m(e, v) \qquad \text{for all }v\in V,
\eeq
where fixed $\tau>0$ is large enough so that by Theorem 3.1 of \cite{RR19} the form
\beq\label{e.aLT} 
a_{\lambda,\tau} := a - \lambda\, b + \tau\, m
\eeq 
is coercive on $V$.
Using the algorithm described in \cite{Au4}, the eigendata for $(a_{\lambda, \tau}\, , \, m)$, i.e. solutions of \eqref{e.Pmutau}, is found to be
\[
\mu +\tau = \tilde{\mu}_k(\lambda, \tau) \qquad \text{and}\qquad e=\tilde{e}_k(\lambda, \tau)\; ,  \qquad k\in \N,
\]
where the eigenvectors are $a_{\lambda, \tau}$-normalized, i.e.
\beq\label{e.aLTnorm}
\|\tilde{e}_k(\lambda, \tau)\|_{a_{\lambda, \tau}}^2 \; := \; \Ac_{\lambda, \tau}( \tilde{e}_k(\lambda, \tau) ) =1.
\eeq 
Rearranging the eigenequation for $(a_{\lambda, \tau}\, , \, m)$ gives that this eigendata satisfies, for fixed $\lambda, \tau$:
\beq\label{e.edata}
a(\tilde{e}_k(\lambda, \tau), v) \;\; = \; \; \lambda\, b(\tilde{e}_k(\lambda, \tau), v) \; + \; (\tilde{\mu}_k(\lambda, \tau) -\tau)\, m(\tilde{e}_k(\lambda, \tau), v) \quad \text{for all } v\in V.
\eeq

For each $i=1,2, \cdots$, define $E_i(\lambda, \tau)$ to be the eigenspace generated by the eigenvectors of $(a_{\lambda, \tau}\, , \, m)$ associated with eigenvalue $\tilde{\mu}_i(\lambda,\tau)$.
Eigenvalues may have multiplicity higher than 1, so if $\tilde{\mu}_1(\lambda, \tau) < \tilde{\mu}_2(\lambda, \tau)= \tilde{\mu}_3(\lambda, \tau)$, then $\tilde{e}_3(\lambda, \tau) \in E_2(\lambda, \tau)$.   
The next result shows that $\tilde{\mu}_k$ is a linear function of $\tau$.

\blem 
Assume $(A1), (A2)$, and $(A3)$ hold, and $\tilde{\mu}_k(\lambda, \tau),  \tilde{e}_k(\lambda, \tau)$, and $E_i(\lambda, \tau)$ are as above for fixed $\lambda, \tau$.
If $\delta \geq 0$, then
\beq
\tilde{\mu}_k(\lambda, \tau+\delta) = \tilde{\mu}_k(\lambda, \tau) + \delta \qquad \text{and}\qquad E_k(\lambda, \tau+ \delta) = E_k(\lambda, \tau).
\eeq
\elem

\bpf
Rearranging the $(a_{\lambda, \tau + \delta}\, , \, m)$-eigenequation for $\tilde{\mu}_k(\lambda, \tau +\delta)\, , \, \tilde{e}_k(\lambda, \tau + \delta)$ gives
\[
a_{\lambda, \tau}(\tilde{e}_k(\lambda, \tau+\delta)\, , \, v) 
\; =\;
 (\tilde{\mu}_k(\lambda, \tau + \delta) - \delta)\, m(\tilde{e}_k(\lambda, \tau + \delta)\, , \, v) \qquad \text{for all }v\in V,
\]
so  $\tilde{\mu}_k(\lambda, \tau + \delta) - \delta \; =\; \tilde{\mu}_i(\lambda, \tau)$ for some $i$.
Separately, adding $\delta\,  m(\tilde{e}_k(\lambda, \tau)\, , \, v)$ to both sides of the $(a_{\lambda, \tau}\, ,\, m)$-eigenequation for $\tilde{\mu}_k(\lambda, \tau)\, , \, \tilde{e}_k(\lambda, \tau)$ leads to
\[
a_{\lambda, \tau + \delta }(\tilde{e}_k(\lambda, \tau)\, , \, v) \; = \; (\tilde{\mu}_k(\lambda, \tau) + \delta)\, m(\tilde{e}_k(\lambda, \tau)\, , \, v) \qquad \text{for all }v\in V,
\]
so $\tilde{\mu}_k(\lambda, \tau) + \delta \; =\; \tilde{\mu}_j(\lambda, \tau + \delta)$ for some $j$.
These two relations give the desired results.
\epf

For the next result, using \eqref{e.edata} define the new values
\beq\label{e.muNoTilde}
\mu_k(\lambda, \tau)\; := \;  \tilde{\mu}_k(\lambda, \tau) -\tau \qquad \text{for each }k\in \N. 
\eeq 

\blem
Assume $(A1), (A2)$ and $(A3)$ hold.  
For fixed $\lambda\in \R$, the new values $\mu_k(\lambda, \tau)$ given in \eqref{e.muNoTilde} are independent of $\tau$.
\elem

\bpf
For $\delta >0$, the linear functional relation proved above gives
\[
\mu_k(\lambda, \tau + \delta) 
\; =\;
 \tilde{\mu}_k(\lambda, \tau + \delta) - (\tau + \delta) 
\; =\;
 \tilde{\mu}_k(\lambda, \tau) + \delta  -  (\tau + \delta)
\; =\;
 \tilde{\mu}_k(\lambda, \tau) - \tau \; =\; \mu_k(\lambda, \tau)
\]
as claimed.
\epf

Thus,  $\mu_k(\lambda, \tau)$ and $E_k(\lambda, \tau)$ are written simply as $\mu_k(\lambda)$ and $E_k(\lambda)$.
Although, the eigenspaces are $\tau$-invariant, the eigenvectors $\tilde{e}_k(\lambda, \tau)$ are, in general, not.
However, since for each $k\in \N$, $E_k(\lam,\gamma)$ is independent of $\gamma$, we may assume that $\tilde{e}_k(\lam, \tau+\delta)=c(\lam,\delta)\tilde{e}_k(\lam,\tau)$, for some constant $c(\lam, \delta)$. To obtain ``cononical'' spectral representations, we define the following:
\beq\label{e.canonicalE}
e_k(\lambda, \tau) \; := \; \sqrt{ \tilde{\mu}_k(\lambda, \tau)}\, \tilde{e}_k(\lambda, \tau).
\eeq
The next lemma shows these scaled eigenvectors satisfy certain norm properties, where 
\beq
\|u\|_m \; := \; \sqrt{ \mathcal{M}(u)} \; = \; \sqrt{ m(u, u)} \qquad \text{defined for all }u\in V,
\eeq 
is called the $m$-norm on $V$ and the $a_{\lambda\tau}$-norm is given above in \eqref{e.aLTnorm}.

\blem Assume $(A1), (A2)$, and $(A3)$ hold, and $\tilde{\mu}_k(\lambda, \tau),  \tilde{e}_k(\lambda, \tau)$, and $E_i(\lambda, \tau)$ are as above for fixed $\lambda, \tau$.
Then
\begin{itemize}
\item[$(i)$]  the $e_k(\lambda, \tau)$ are $m$-normalized: \; $ \|e_k(\lambda, \tau)\|_m \; = \; 1$.

\vspace{1em}
\item[$(ii)$] for $\delta >0$, \; $\|e_k(\lambda, \tau )\|_{a_{\lambda, \tau +\delta}} \; =\;  \sqrt{  \tilde{\mu}_k(\lambda, \tau) + \delta}$.
\end{itemize}
\elem

\bpf
The first assertions follows from 
\[
\|e_k(\lambda, \tau)\|_m^2 \; =\; \Msc( e_k(\lambda, \tau)) \; =\;\tilde{\mu}_k(\lambda, \tau)\Msc( \tilde{e}_k(\lambda, \tau)) \; =\; \mathscr{A}_{\lambda, \tau}(\tilde{e}_k(\lambda, \tau)) \; =\; 1,
\]
using the $(a_{\lambda, \tau}\, , \, m)$-eigenequation.  The same eigenequation also gives
\[
\|e_k(\lambda, \tau)\|_{a_{\lambda, \tau+\delta}}^2 
\; =\;
 \Asc_{\lambda, \tau+\delta}( e_k(\lambda, \tau)) 
\; =\;
\mathscr{A}_{\lambda, \tau}(e_k(\lambda, \tau)) + \delta\, \Msc( e_k(\lambda, \tau))
\; =\;
\tilde{\mu}_k(\lambda, \tau) + \delta
\] 
so the second assertion holds.
\epf

\begin{rem}
Since $\tilde{e}_k(\lambda, \tau +\delta)$ has $a_{\lambda, \tau + \delta}$-norm one,  part $(ii)$ of this lemma and equation \eqref{e.canonicalE} imply 
\beq\label{e.epsilonE}
\tilde{e}_k(\lambda, \tau + \delta) 
\; =\;
 \frac{  \sqrt{  \tilde{\mu}_k(\lambda, \tau)}    }{    \sqrt{  \tilde{\mu}_k(\lambda, \tau) + \delta }   } \; \tilde{e}_k(\lambda, \tau)
\eeq
\end{rem}

\blem Assume $(A1)$, $(A2)$, and $(A3)$ hold.
Then the $e_k(\lambda, \tau)$ given in \eqref{e.canonicalE} are independent of $\tau$.
\elem

\bpf
Multiplying equation \eqref{e.epsilonE} by $\sqrt{ \tilde{\mu}_k(\lambda, \tau) + \delta}$ gives
\[
e_k(\lambda, \tau + \delta) \; =\; \sqrt{ \tilde{\mu}_k(\lambda, \tau + \delta)}\, \tilde{e}_k(\lambda, \tau + \delta) \; =\; \sqrt{ \tilde{\mu}_k(\lambda, \tau)}\, \tilde{e}_k(\lambda, \tau) \;  =\; e_k(\lambda, \tau)
\]
and so the claim holds.
\epf

This justifies writing $e_k(\lambda)$ instead of $e_k(\lambda, \tau)$.
The pairs
\beq\label{e.canonical}
(\mu_k(\lambda)\, , \, e_k(\lambda))\qquad \text{with }k\in\N,
\end{equation}
are regarded as the {\it canonical spectral data} for the triple $(a, b, m)$.

\vspace{1em}
\section{\bf Two-parameter Fredholm Alternative}\label{s.LinearFredholm}

This and the next section treat the linear version of \eqref{E.Non1} stated as: 
for fixed $(\lambda, \mu)\in \R$ and  $\ell \in V^*$  consider the problem of finding $u$ satisfying
\beq\label{e.LnoR}
a(u,v)  - \lambda \, b(u,v) - \mu\, m(u,v) \;  \; = \;\; \ell(v) \qquad \text{for all }v\in V.
\eeq
This equation is called the linear {\it Fredholm equation} for $(a, b, m)$.
The result for the nonresonance case; that is, the case where $(\lam, \mu)\not\in\sigma(a,b,m)$, is the following.

\btm\label{E:NonSpec} Assume $(A1)$, $(A2)$, and $(A3)$ hold and suppose $(\lambda, \mu) \notin \sigma(a, b, m)$.  
Then $\mu\neq \mu_k(\lambda)$ for all $k$, and there is a unique solution $\hat{u}$ to problem \eqref{e.LnoR} which has the spectral representation
\beq 
\hat{u} \; = \; \sum_{k=1}^\infty   \frac{ 1 }{\mu_k(\lambda) - \mu}  \, \ell(e_k(\lambda)) \, e_k(\lambda).
\eeq 
\etm

The proof of this also yields the following estimates for the unique solution that are particularly useful in the applications.

\btm\label{E:NonSpec2} Assume $(A1)$, $(A2)$, and $(A3)$ hold and suppose $(\lambda, \mu)\notin \sigma(a, b, m)$.
Let $\tau>0$ be large enough so that $a_{\lambda \tau}$ in \eqref{e.aLT} is coercive, and denote by $\|\cdot \|_{a_{\lambda \tau}^*}$ the dual norm on $V^*$ when $V$ is equipped with the norm $\| u \|_{a_{\lambda, \tau}} := \sqrt{ a_{\lambda \tau}(u, u) }$. 
	\begin{itemize}
	\item[$(i)$] When $\mu < \mu_1(\lambda)$, then the unique solution $\hat{u}$ to \eqref{e.LnoR} satisfies the bound 
\beq\label{e.LinearBelowOne}
\|\hat{u}\|_{a_{\lambda\tau}} \; \leq \; \left(   \frac{\mu_1(\lambda) + \tau}{\mu_k(\lambda) -\mu}    \right) \, \|\ell\|_{a_{\lambda\tau}^*}
\eeq
	\item[$(ii)$]  When $\mu_{k_0}(\lambda) < \mu < \mu_{k_0+1}(\lambda)$, then the unique solution $\hat{u}$ to \eqref{e.LnoR} satisfies the bound
\beq\label{e.LinearBetweenK}
\|\hat{u}\|_{a_{\lambda\tau}} \; \; \leq \; \;
 \frac{(\mu + \tau)( \mu_{k_0+1}(\lambda) - \mu_1(\lambda)}{  (\mu_1(\lambda) + \tau) (\mu_{k_0+1}(\lambda) - \mu)}\,  \| P_{k_0}\hat{u} \|_{a_{\lambda \tau}} 
\; +\; 
 \left(  \frac{\mu_{k_0+1}(\lambda) + \tau}{ \mu_{k_0+1}(\lambda) - \mu} \right) \|\ell\|_{a_{\lambda\tau}^*} 
\eeq
where 
\beq 
P_{k_0}\hat{u} \;  :=  \; \sum_{k=1}^{k_0}\frac{\ell(e_k(\lambda))}{\mu_k(\lambda) - \mu} e_k(\lambda).
\eeq 
	\end{itemize}
\etm

\bpf The following argument yields the results for both theorems.

Rewrite  equation \eqref{e.LnoR} to get
\beq
a(u, v) - \lambda \, b(u, v) + \tau\, m(u, v) - (\mu + \tau)\, m(u, v)  \; \; = \; \; \ell(v) \qquad \text{for all }v\in V,
\eeq
where $\tau>0$ is large enough so that $a_{\lambda,\tau} := a - \lambda\, b + \tau\, m$ is coercive.
Rewritten once more this reads as
\beq\label{e.abmAuch}
a_{\lambda, \tau}(u, v) - (\mu + \tau)\, m(u, v) \; \; = \; \; \ell(v) \qquad \text{for all }v\in V.
\eeq
Since $\mu\neq \mu_k(\lambda)$ we also have $\mu + \tau \neq \tilde{\mu}_k(\lambda, \tau)$ for all $k$.
Then Theorem 10.1 of \cite{Au4} shows that the unique solution $\hat{u}$ of \eqref{e.abmAuch} is 
\[
\hat{u} \; = \; \sum_{k=1}^\infty \frac{\tilde{\mu}_k(\lambda, \tau)}{ \tilde{\mu}_k(\lambda, \tau) - (\mu + \tau)} \, \ell(\tilde{e}_k(\lambda, \tau)) \, \tilde{e}_k(\lambda, \tau)
\]
Splitting $\tilde{\mu}_k(\lambda, \tau)$ as $\sqrt{ \tilde{\mu}_k(\lambda, \tau)} \sqrt{ \tilde{\mu}_k(\lambda, \tau)} $ and using the linearity of $\ell$ and the scaling \eqref{e.canonicalE} for $e_k(\lambda)$, together with \eqref{e.edata} for $\mu_k(\lambda)$, leads to desired spectral representation for $\hat{u}$.

 Substituting $u=v=\hat{u}$ in \eqref{e.abmAuch} gives
\beq
\sum_{k=1}^\infty \left( 1 - \frac{\mu+\tau}{\tilde{\mu}_k(\lambda, \tau)}\right) \tilde{c}_k^2 \; \; =\; \; \ell(\hat{u})
\eeq
upon applying Parseval's equality $\|\hat{u}\|_{a_{\lambda, \tau}}^2 \; = \; \sum_{k=1}^\infty \tilde{c}_k^2$ with respect to the basis $\{\tilde{e}_k(\lambda, \tau)\}$.
When $\mu < \mu_1(\lambda)$, then  $\left( 1 - \frac{\mu+\tau}{\mu_1(\lambda) + \tau}\right) \leq \left( 1 - \frac{\mu+\tau}{ \mu_k(\lambda)+ \tau} \right)$ for all $k$.
Hence,
\[
\left( 1 - \frac{\mu+\tau}{\mu_1(\lambda) + \tau}\right) \sum_{k=1}^\infty \tilde{c}_k^2 \; \leq  \; \sum_{k=1}^\infty  \left( 1 - \frac{\mu+\tau}{ \mu_k(\lambda)+ \tau} \right) \tilde{c}_k^2 \; =\; \ell(\hat{u})
\]
The extreme sides of this relation lead to
\[
\left( 1 - \frac{\mu+\tau}{\mu_1(\lambda) + \tau}\right) \|\hat{u}\|_{a_{\lambda, \tau}}\; \leq  \; \;  \; \ell\left( \frac{\hat{u}}{ \|\hat{u}\|_{a_{\lambda, \tau} } } \right)
\]

Taking the supremum of $\ell(u)$ over all $\|u\|_{a_{\lambda, \tau}}=1$, and rearranging  gives the bound in \eqref{e.LinearBelowOne}.

When $\mu_{k_0}(\lambda)\, < \,  \mu  \,  < \, \mu_{k_0+1}(\lambda)$, then    $\left( 1 - \frac{\mu+\tau}{\mu_{k_0+1}(\lambda) + \tau}\right) \leq \left( 1 - \frac{\mu+\tau}{ \mu_k(\lambda)+ \tau} \right)$ for all $k\geq k_0+1$.  
This gives
\[
\left( 1 - \frac{\mu+\tau}{ \mu_{k_0+1}(\lambda) + \tau} \right) \sum_{k=1}^\infty \tilde{c}_k^2 \; \leq \; 
\left( 1 - \frac{\mu+\tau}{ \mu_{k_0+1}(\lambda) + \tau} \right) \sum_{k=1}^{k_0} \tilde{c}_k^2 \; + \; 
 \sum_{k=k_0+1}^\infty \left( 1 - \frac{\mu+\tau}{ \mu_k(\lambda) + \tau} \right)\tilde{c}_k^2.
\]
Adding and subtracting $\sum_{k=1}^{k_0}  \left( 1 - \frac{\mu+\tau}{ \mu_k(\lambda) + \tau} \right)\tilde{c}_k^2$ the left side becomes
\[
\left( 1 - \frac{\mu+\tau}{ \mu_{k_0+1}(\lambda) + \tau} \right) \sum_{k=1}^{k_0} \tilde{c}_k^2
 \; - \; 
 \sum_{k=1}^{k_0} \left( 1 - \frac{\mu+\tau}{ \mu_k(\lambda) + \tau} \right)\tilde{c}_k^2
 \; + \; 
 \sum_{k=1}^\infty \left( 1 - \frac{\mu+\tau}{ \mu_k(\lambda) + \tau} \right)\tilde{c}_k^2.
\]
This expression is majorized by
\[
\left( 1 - \frac{\mu+\tau}{ \mu_{k_0+1}(\lambda) + \tau} \right) \sum_{k=1}^{k_0} \tilde{c}_k^2
 \; - \; 
 \left( 1 - \frac{\mu+\tau}{ \mu_1(\lambda) + \tau} \right) \sum_{k=1}^{k_0} \tilde{c}_k^2
 \; + \; 
\ell(\hat{u}).
\]
since the series is equal to $\ell(\hat{u})$ and $-\left( 1 - \frac{\mu+\tau}{\mu_k(\lambda) + \tau}\right) \leq -\left( 1 - \frac{\mu+\tau}{\mu_1(\lambda) + \tau}\right)$ for all $k=1, \ldots k_0$.
Combining the two sums and simplifying thus shows that
\[
\left( 1 - \frac{\mu+\tau}{ \mu_{k_0+1}(\lambda) + \tau} \right) \sum_{k=1}^\infty \tilde{c}_k^2 \; \;  \leq \;  \; 
\frac{  (\mu+\tau)[ \mu_{k_0+1}(\lambda) - \mu_1(\lambda)  ] }{  (\mu_{k_0+1}(\lambda) + \tau)(\mu_1(\lambda) + \tau)  } \, \sum_{k=1}^{k_0} \tilde{c}_k^2 \; + \; \ell(\hat{u}).
\]
Rearrange this inequality to get
\[
\|\hat{u}\|_{a_{\lambda\tau}}^2  \; \; \leq \; \; \frac{  (\mu+\tau)[\mu_{k_0+1}(\lambda) - \mu_1(\lambda)]  }{  (\mu_1(\lambda)+\tau  )(  \mu_{k_0+1}(\lambda) - \mu)   } \|P_{k_0}\hat{u}\|_{a_{\lambda\tau}}^2 \; + \; \left(   \frac{ \mu_{k_0+1}(\lambda) + \tau  }{  \mu_{k_0+1}(\lambda) - \mu  }  \right)\, \ell(\hat{u}),
\]
where $\|P_{k_0}\hat{u}\|_{a_{\lambda\tau}}^2 = \sum_{k=1}^{k_0}\tilde{c}_k^2$ is the norm (squared) of the projection $P_{k_0}\hat{u}$.
Inequality \eqref{e.LinearBetweenK} follows from this by dividing by $\|\hat{u}\|_{a_{\lambda\tau}}$, using that $\frac{   \|P_{k_0}\hat{u}\|_{a_{\lambda\tau}}   }{ \|\hat{u}|_{a_{\lambda\tau}}  } \leq 1$ and taking the supremum over all $\|u\|_{a_{\lambda\tau}}=1$ to get the dual norm of $\ell$ with respect to the $a_{\lambda\tau}$ norm.
\epf

We finish this section with a Fredholm alternative in the case of resonance, that is, the case where $(\lambda, \mu)\in \sigma(a, b, m)$.
The result for the two-parameter Fredholm Equation \eqref{e.LnoR} is this case is stated in the next theorem, where for fixed $(\lambda, \mu)\in \sigma(a, b, m)$ the notation 
\beq
J_{\lambda, \mu} \;  :=  \; \{ k\in \N : \mu_k(\lambda) = \mu \} \qquad \text{and}\qquad E_{\lambda,\mu} \; := \; {\rm span} \{ e_k(\lambda) : k\in J_{\lambda,\mu} \}  
\eeq 
is used instead of the previous $E_k(\lambda)$ that denoted the distinct eigenspaces of $(a_{\lambda}\, , \, m)$; so $E_{\lambda,\mu} = E_k(\lambda)$ for some $k$.

\btm\label{E:ResFred}  Assume $(A1)$, $(A2)$, and $(A3)$ hold, and let $(\lambda, \mu)\in \sigma(a, b, m)$ and  $\ell \in V^*$ be fixed. 
Then $\mu  = \mu_{k_0}(\lambda)$ for some $k_0$ and 
\begin{itemize}
\item[$(i)$]\quad there are solutions $\hat{u}$ to \eqref{e.LnoR} if and only if $\ell(e) = 0$ for all $e\in E_{\lambda, \mu}$.
\item[$(ii)$]\quad the solution set $S_{\lambda,\mu}$ to \eqref{e.LnoR} is the affine subspace comprised of all $\hat{u}$ of the form 
\beq\label{e.LRsol}
\hat{u} \; =\; \sum_{ k\notin J_{\lambda,\mu}  } \frac{1}{\mu_k(\lambda) - \mu} \, \ell(e_k(\lambda))\, e_k(\lambda) \; + \; \hat{v}
 \qquad \text{where }\quad 
\hat{v}\in E_{\lambda, \mu}.
\eeq
\end{itemize}
\etm

\bpf
When $\ell(e) \neq 0$ for some $e\in E_{\lambda, \mu}$, then equation \eqref{e.LnoR} does not hold for $v= e$.
Now suppose $\ell(e)=0$ for all $e\in E_{\lambda, \mu}$.
Substituting $\hat{u}$ from \eqref{e.LRsol} into the left side of \eqref{e.LnoR} simplifies to
\[
a\left( \sum_{k\notin J_{\lambda,\mu}} \ell(e_k(\lambda))\, e_k(\lambda)\; , \; v\right) \;\; = \; \;\ell(v)
\]
which implies that $\hat{u}$ is a solution, so assertion $(i)$ holds.
Since this holds for any $\hat{v}\in E_{\lambda,\mu}$, hence any $\hat{u}$ as in \eqref{e.LRsol}, assertion $(ii)$ also holds.
\epf

The  special solution $\hat{u}_0$ obtained by taking $\hat{v}$ in \eqref{e.LRsol}  to be zero is often used in practice since then $\hat{u}$ has additional orthogonality structure.




\vspace{1em}
\section{\bf Nonresonance}\label{s.Nonresonance}

With Theorem \ref{E:NonSpec} and \ref{E:ResFred} in hand, we now turn our attention to the main interest of this paper, which is the study of the nonlinear bilinear form problem \eqref{E.Non1}. 
This problem is a nonlinear perturbation of \eqref{e.LnoR} and as such, the Fredholm alternatives developed in the previous section will play an important role in our analysis.  
For ease of reference, we remind the reader that we are now considering the problem of finding $u$ satisfying
\beq\label{e.nonresonance}
a(u, v) \; =\; \lambda\, b(u, m) + \mu\, m(u, v) + \varepsilon\, F(u, v) \qquad \text{for all }v\in V.
\eeq
This will be called the {\it nonlinear Fredholm equation} for $(a, b, m)$.
At the moment we will be focusing on the nonresonant case; that is, the case in which $(\lambda, \mu) \notin \sigma(a, b, m)$. 
Throughout,  $\varepsilon \in \R $ will be fixed, and we will suppose that $F$ is such that $F(u, \cdot) \in V^*$ for all $u\in V$.
We present two results in this context which give the existence of solutions to \eqref{e.nonresonance} under assumptions on the mapping $u\to F(u,\cdot)$ from $V$ to $V^*$.

\begin{thm}\label{T:Nonres}
Suppose that $(\lambda, \mu) \notin \sigma(a, b, m)$ and that $u\to F(u,\cdot)$ is continuously differentiable.  Then there exists a $\delta>0$ such that if $-\delta<\ve<\delta$, then \eqref{e.nonresonance} has a locally unique solution $u_\ve$. Further, the mapping $u:(-\delta, \delta)\to V$, defined by $u(\ve)=u_\ve$, is continuously differentiable with $\ds\lim_{\ve\to 0}u(\ve)=0$.
\end{thm}
\begin{proof}
Define the map $G:\R\times V \rightarrow V$ by
\beq\label{E:Nonlres}
G(\varepsilon, u) \; \; = \; \; u \; - \; \sum_{k=1}^\infty \frac{1}{\mu_k(\lambda) - \mu} \varepsilon F(u, e_k(\lambda))\, e_k(\lambda)
\eeq
From the spectral representation of Theorem \ref{E:NonSpec}, the vector
\beq
w_u \; \; := \; \; \sum_{k=1}^\infty \frac{1}{\mu_k(\lambda) - \mu} \varepsilon F(u, e_k(\lambda)) \, e_k(\lambda)
\eeq
satisfies
\beq
a(w_u\, , \, v) \; = \; \lambda\, b(w_u\, , \, v) + \mu\, m(w_u\, , \, v) + \varepsilon F(u, v) \qquad \text{for all }v\in V.
\eeq
If $G(\varepsilon, u)=0$, then we get $u = w_u$ and this last eigenequation simplifies to equation \eqref{e.nonresonance}.
Thus, the solutions of \eqref{e.nonresonance} are precisely the zeros of $G$. 

Let $D_uF$ denote the derivative of $u\to F(u,\cdot)$.
Then
\[
D_uG(\varepsilon, u)[v] \; \; = \; \; v \; - \; \varepsilon \sum_{k=1}^\infty \frac{1}{\mu_k(\lambda) - \mu} D_uF(u, e_k(\lambda))[v]\, e_k(\lambda).
\]
It follows that $D_uG(0, 0) = I$, where $I$ is the identity operator on V.
Since $G(0, 0) =0$, the {\it Implicit Function Theorem} implies there exists a $\delta>0$ and a continuously differentiable mapping $u:(-\delta, \delta)\to V$ such that $G(\ve, u(\ve))=0$ for each $\ve\in (-\delta,\delta)$and $\ds\lim_{\ve\to 0}u(\ve)=0$.  This completes the proof.
\end{proof}

We finish this section with a result which requires slightly weaker assumptions on the nonlinearity $F$.  As a consequence we loose the differentiabilty of the solution $\ve\to u(\ve)$ that we had above.

\begin{thm}
Suppose that $(\lambda, \mu) \notin \sigma(a, b, m)$ and that that $u\to F(u,\cdot)$ is a Lipschitz mapping. Then there exists a $\delta>0$ such that if $-\delta<\ve<\delta$, then \eqref{e.nonresonance} has a unique solution $u_\ve$.
\end{thm}
\begin{proof}
Suppose that the mapping $u\to F(u,\cdot)$ is Lipschitz, with Lipschitz constant $\kappa$ with respect to the $m$-norm. Define 
\beq
H(\varepsilon, u ) \; \; = \; \; \varepsilon \sum_{k=1}^\infty \frac{1}{\mu_k(\lambda) - \mu} F(u, e_k(\lambda))\, e_k(\lambda).
\eeq
We then have
\begin{align*}
\| H(\varepsilon, u) - H(\varepsilon, w)\|_m^2 \; \; 
= & \; \; \varepsilon^2 \left\| \sum_{k=1}^\infty \frac{1}{ \mu_k(\lambda) - \mu } \left( F(u, e_k(\lambda)) \; - \; F(w, e_k(\lambda)) \right)\, e_k(\lambda)  \right\|_m^2\\
= & \; \; \varepsilon^2\sum_{k=1}^\infty  \frac{1}{ |\mu_k(\lambda) - \mu|^2 } \| F(u, e_k(\lambda)) \; - \; F(w, e_k(\lambda)) \|_m^2 \\
\leq & \; \; \varepsilon^2 \sum_{k=1}^\infty \frac{1}{|\mu_k(\lambda) - \mu|^2}\, \kappa\, \|u - w\|_m^2
\end{align*}
which is a contraction for small enough $\varepsilon >0$. The result now follows from the {\it Contraction Mapping Theorem}.
\end{proof}

\vspace{1em}
\section{\bf Resonance}\label{s.Resonance}

In this section we continue our study of problem \eqref{e.nonresonance}; that is, we are again interested in finding elements $u\in V$ such that 
\beqq
a(u, v) \; =\; \lambda\, b(u, m) + \mu\, m(u, v) + \varepsilon\, F(u, v) \qquad \text{for all }v\in V.
\eeqq
Our focus in this section, however, will be on the case of resonance; that is, the case in which $(\lambda, \mu) \in  \sigma(a, b, m)$. The solvability of \eqref{e.nonresonance} in this case of resonance is more delicate and requires a more careful analysis than was needed in the nonresonant case. Note, in particular, that the mapping $G$ defined in \eqref{E:Nonlres} is no longer well-defined.  Thus, analysis of \eqref{e.nonresonance} in this resonant case must proceed along a different route. The results we establish provide a substantial generalization of similar ideas found in \cite{Lewis,Mar2,Mar3}.

As before, $\varepsilon \in \R$ will be fixed, and we will suppose that $F$ is such that $F(u, \cdot) \in V^*$ for all $u\in V$.  
We will again assume that the mapping $u\to F(u,\cdot)$ is continuously differentiable, and in order to simplify the statement of the theorem below, the notation
\[
J_{\lambda, \mu} \; := \; \{k\in \N :  \mu_k(\lambda)=\mu\} \qquad \text{and}\qquad E_{\lambda,\mu} \; := \; \text{span}\{e_k(\lambda) :  k\in J_{\lambda,\mu}\}
\]
that was used, for fixed $(\lambda, \mu)\in \sigma(a,b,m)$, in the linear two-parameter Fredholm alternative at resonance will also be used here.
Finally, for every $w\in V$, define $\Lambda_w:V\to E_{\lam, \mu}$ by
\beqq
\Lambda_w(v)\; := \; \sum_{k\in J_{\lam,\mu}}D_uF(w,e_k(\lam))[v]e_k(\lam).
\eeqq

\begin{thm}\label{T:Res}
Suppose that $(\lambda,\mu)\in \sigma(a, b, m)$ and that $u\to F(u,\cdot)$ is continuously differentiable. Suppose further that there exists $u_0\in E_{\lam,\mu}$ with $F(u_0, e_k(\lam))=0$ for each $k\in J_{\lam, \mu}$. If the restriction of $\Lambda_{u_0}$ to $E_{\lam, \mu}$ is an injection, then \eqref{e.nonresonance} has a solution for small $\ve$.
\end{thm}

\begin{proof}
In this resonance case, define the map $G:\R\times V \rightarrow V$ by
\beq
G(\varepsilon, u) \; \; = \; \; u \; - \; \sum_{k\in J_{\lambda, \mu}} [m(u, e_k(\lambda)) - F(u, e_k(\lambda))]e_k(\lambda) - \varepsilon \sum_{k\notin J_{\lambda, \mu}} \frac{1}{\mu_k(\lambda) - \mu}  F(u, e_k(\lambda))\, e_k(\lambda).
\eeq

From the spectral representation, Theorem \ref{E:ResFred}, the vector 
\beq
w_u \; = \; \sum_{k\in J_{\lambda, \mu}} [ m(u, e_k(\lambda)) - F(u, e_k(\lambda))]\, e_k(\lambda) \; + \; \varepsilon \sum_{k\notin J_{\lambda, \mu}} \frac{1}{\mu_k(\lambda) - \mu} F(u, e_k(\lambda))\, e_k(\lambda)
\eeq
satisfies
\[
a(w_u\, , \, v) \; = \; \lambda\, b(w_u\, , \, v) + \mu\, m(w_u\, , \, v) + \varepsilon F(u, v) \qquad \text{for all }v\in V
\]
if and only if
\[
F(u, e_k(\lambda)) \; =\; 0 \qquad \text{for all }k\in J_{\lambda, \mu}.
\]
However, if $G(\varepsilon, u) =0$, then $u=w_u$ and from the uniqueness of the orthogonal decomposition
\[
u\; =\; \sum_{k\in J_{\lambda, \mu}} m(u, e_k(\lambda))\, e_k(\lambda) \; + \; \sum_{k\notin J_{\lambda, \mu}} m(u, e_k(\lambda))\, e_k(\lambda)
\]
we get that
\[
\sum_{k\in J_{\lambda, \mu}} F(u, e_k(\lambda)) \, e_k(\lambda) \; \; = \; \; 0 \qquad \text{or that}\qquad F(u, e_k(\lambda)) \; = \; 0 \quad \forall k\in J_{\lambda, \mu},
\]
so that the zeros of $G$ correspond to the solutions of \eqref{e.nonresonance}.

Now suppose $u_0        \in E_{\lam,\mu}$ and that $F(u_0, e_k(\lambda)) =0$ for all $k\in J_{\lambda, \mu}$.
Under this assumption, $G(0, u_0) \; = \; 0$.
Furthermore, 
\begin{align*}
D_uG(\varepsilon, u)[v] \; = \; v \;  - \;  \sum_{k\in J_{\lambda, \mu}} m(v,e_k(\lambda))e_k(\lambda) \; - & \;  \sum_{k\in J_{\lambda, \mu}} D_uF(u,e_k(\lam))[v]e_k(\lambda) \\
 -\;  &  \varepsilon \sum_{k\notin J_{\lambda, \mu}} \frac{1}{\mu_k(\lambda) - \mu} D_uF(u, e_k(\lambda))[v]e_k(\lambda)
\end{align*}
so that
\begin{align*}
D_uG(0, u_0)[v] \; =& \; v \, - \, \sum_{k\in J_{\lambda, \mu}} m(v,e_k(\lambda)) e_k(\lambda) \; - \; \sum_{k\in J_{\lambda, \mu}} D_uF(u_0, e_k(\lambda))[v]e_k(\lambda) \\
=& \; \sum_{k\notin J_{\lambda, \mu}} m(v, e_k(\lambda))e_k(\lambda) \; -\; \sum_{k\in J_{\lambda,\mu}} D_uF(u_0, e_k(\lambda))[v]e_k(\lambda)\\
=&\sum_{k\notin J_{\lambda, \mu}} m(v, e_k(\lambda))e_k(\lambda)\;-\; \Lambda_{u_0}(v).
\end{align*}

Since
$\sum_{k\notin J_{\lambda, \mu}} m(v, e_k(\lambda))\, e_k(\lambda)$ and $\Lambda_{u_0}(v)$ are orthogonal, it follows that if $D_uG(0, u_0)[v] \; =\; 0$, then
\[
\sum_{k\notin J_{\lambda, \mu}} m(v, e_k(\lambda))\, e_k(\lambda) \; = \; 0 \qquad \text{and}\qquad \Lambda_{u_0}(v)\; = \; 0.
\]
The first equation implies $v\in E_{\lam,\mu}$, but then the second implies $v=0$, since $\Lambda_{u_0}$ is injective on $E_{\lam,\mu}$.  It follows that $D_uG(0,u_0)$ is injective.

We now show that $D_uG(0,u_0)$ is a surjection. To see this, let $h\in V$ and define 
\[
\begin{split}
v_1 \; &:= \; \sum_{k\in J_{\lambda, \mu}} \left(   m(h, e_k(\lambda))  \;+\;  D_uF(u_0, e_k(\lambda)) \left[  \sum_{k\notin J_{\lambda,\mu}} m(h, e_k(\lambda))e_k(\lambda)  \right]  \right)\, e_k(\lambda)\\&=\sum_{k\in J_{\lam,\mu}}m(h,e_k(\lam))e_k(\lam)+\Lambda_{u_0}\left(\sum_{k\not\in J_{\lambda, \mu}} m(h, e_k(\lam))\, e_k(\lam)\right).
\end{split}
\]
Since $E_{\lam,\mu}$ is finite dimensional and $\Lambda_{{u_0} \big|_{E_{\lam,\mu}}}$ is injective, it is actually a bijection. Thus, there is a unique $v_0\in E_{\lam, \mu}$ with $\Lambda_{u_0}(v_0) = v_1$. 
Define
\[
v_h \; :=\; \sum_{k\notin J_{\lambda, \mu}} m(h, e_k(\lambda))\, e_k(\lambda) \; -\; v_0.
\]
From the following calculation, suppressing the dependence on $\lambda$, we get 
\begin{align*}
D_uG(0, u_0)[v_h] \; =& \; \sum_{k\notin J_{\lambda, \mu}} m(v_h, e_k)\, e_k \; - \; \Lambda_{u_0}(v_h)\\
=& \; \sum_{k\notin J_{\lambda, \mu}} m(h, e_k)\, e_k \; - \; \Lambda_{u_0}\left(   \sum_{k\not\in J_{\lambda, \mu}} m(h, e_k)\, e_k   \right) \; + \; \Lambda_{u_0}(v_0) \\
=&\sum_{k\notin J_{\lambda, \mu}} m(h, e_k)\, e_k \; - \; \Lambda_{u_0}\left(   \sum_{k\not\in J_{\lambda, \mu}} m(h, e_k)\, e_k   \right) \; + \; v_1 \\
=& \; \sum_{k\notin J_{\lambda, \mu}} m(h, e_k)\, e_k \; + \; \sum_{k\in J_{\lambda, \mu}} m(h, e_k)\, e_k \\
=& \; \;  h,
\end{align*}
so that $D_uG(0, u_0)$ is surjective.
The {\it Open Mapping Theorem} now gives that $D_uG(0, u_0)$ has a continuous inverse.
Therefore, by the {\em Implicit Function Theorem}, solutions to \eqref{e.nonresonance} exist for small enough $\ve$. This completes the proof.
\end{proof}

\begin{rem}
If we refer to the solutions obtained in Theorem \ref{T:Res} as $u_0(\ve)$, then the mapping $\ve\to u_0(\ve)$ is again locally unique, continuously differentiable,  and we have $\ds\lim_{\ve \to 0}u_0(\ve)=u_0$.
\end{rem}

\vspace{1em}
\section{Example}
\subsection{\bf Steklov-Robin eigenproblems and Steklov-Robin Fredholm equations}\label{s.SRFE}

To exemplify our abstract results, this section studies the (weak) solvability of the following elliptic equation with specific nonlinear reaction terms $f, g$ stated below:

\beq\label{e.SRN}
\begin{cases} 
-{\rm div}(A\nabla u)\, + \, c(x)\,u \; = \; \mu\, m_0(x)\, u \, + \, \varepsilon \,f(x, u) \qquad \text{in }\Omega \\
 \\
 (A\nabla u)\cdot \nu \, + \, b_c(x)\, u \; = \; \lambda\, b_0(x)\, u \, + \, \varepsilon\, g(x,u) \qquad \text{on }\partial\Omega 
\end{cases}.
\eeq
When $\varepsilon \in \R$ is zero, this is called a {\it Steklov-Robin eigenproblem} on $\Omega$, and $(\lambda, \mu)\in \R^2$ is said to be a {\it Steklov-Robin eigenpair} when \eqref{e.SRN} has a nontrivial solution.
When $\varepsilon \neq 0$, equation \eqref{e.SRN} will be called a {\it Steklov-Robin Fredholm equation}.

Here, $\Omega \subset \R^N$ is a bounded domain with boundary $\partial \Omega$ of class $C^{0,1}$,  $\nu$ is the outward unit normal to $\partial \Omega$ so that $(A\nabla u)\cdot \nu$ is the outward unit conormal derivative and the given data of the equation is assumed to satisfy:\\

\noindent
{\bf (C1):}\quad $A(x):=(a_{ij}(x))$ is a real symmetric matrix whose components are bounded Lebesgue-measurable functions on $\Omega$ and there are constants $0< \kappa_0 \leq  \kappa_1$ satisfying
\[
\kappa_0|\xi|^2 \; \leq \langle \, A(x)\xi\, , \, \xi\, \rangle \; \leq \; \kappa_1|\xi|^2 \qquad \text{for all }\xi\in \R^N\, ,\; x\in \Omega;
\]

\noindent 
 {\bf (C2):}\quad $c\geq 0$ on $\Omega$ and $c\in L^p(\Omega)$ for some $p\geq N/2$ when $N\geq 3$, or $p>1$ when $N=2$, where $L^p(\Omega)$ are the usual Lebesgue spaces on $\Omega$;
 
 \noindent 
 {\bf (C3):}\quad $b_c\geq 0$ on $\partial\Omega$ with $b\in L^{\infty}(\partial\Omega)$ and $c$ and $b_c$ satisfy the compatibility condition:
 \[
 \int_\Omega c\, dx \, + \, \int_{\partial\Omega} b_c\, d\sigma \; \; > \; \; 0,
 \]
 where $\sigma$ is Hausdorff measure on $\partial\Omega$;
 
 \noindent  
 {\bf (C4):}\quad $m_0>0$ on $\Omega$ with $m_0\in L^q(\Omega)$ for some $q>N/2$, and satisfies
 \[
 \int_\Omega m_0\, u^2\, dx \; > \; 0 \qquad \text{for all non-zero } u\in H^1(\Omega),
 \]
 where $H^1(\Omega)$ denotes the usual real Sobolev space of functions on $\Omega$;
 
 \noindent 
 {\bf (C5):}\quad $b_0\in L^\infty(\partial\Omega)$ and $b_0$ may be sign-changing on $\partial\Omega$.\\

 \noindent 
 The specific nonlinearities considered in this example are constructed as follows:\\
 
 \noindent
 {\bf (Cf):}\quad  Let $f_0\in L^2(\Omega)$ be such that $\int_{\Omega} f_0(x)\, dx >0$, and for $2< p < \infty$ define
\[
\begin{cases} 
\ell_f(u) \;  :=  \;\int_{\Omega} f_0\, u\, dx \\
   \\
f(x,u) \;  := \;  |\ell_f(u)|^{p-2}\ell_f(u)f_0(x) \qquad \text{for }u\in \Hone\text{ and } x\in \Omega. 
\end{cases} 
\]
 
 \noindent 
 {\bf (Cg):}\quad Let $g_0\in L^2(\partial\Omega)$ be such that $\int_{\partial\Omega} g_0(x)\, d\sigma >0$, and for $2< p < \infty$ define
\[
\begin{cases} 
\ell_g(u) \; := \;  \int_{\partial\Omega} g_0\,u \, d\sigma\\
   \\
g(x,u) \; :=\;  |\ell_g(u)|^{p-2}\ell_g(u)g_0(x) \qquad \text{for }u\in \Hone \text{ and } x\in 
\partial\Omega.
\end{cases} 
\]

 \begin{rem}
When similar operators $f, g$  in (Cf) and (Cg) are taken to be (appropriately) defined on the Banach-Sobolev space $W^{1, p}(\Omega)$, then the analysis of Auchmuty \cite{Au05} leads to very general versions of Poincar\'{e}- and Friedrichs-type inequalities.  
So these $f, g$ operators are of importance in the analysis of partial differential equations on Sobolev spaces.\\
 \end{rem}
 
 \begin{rem} 
 Readers interested in more general nonlinearities (in the interior and on the boundary equations) may consult Mavinga and Nkashama \cite{MN10}, where $f, g$ are assumed to satisfy: 
 
 \noindent
 {\bf (Cf'):}\quad $f\in C(\overline{\Omega}\times \R)$ and there are constants $b_1, b_2>0$ such that
 \[
 |f(x,u)|\; \leq \; b_1 + b_2|u|^s \qquad \text{with } 0\leq s< \frac{N+2}{N-2};
 \]
 
 \noindent
 {\bf (Cg'):}\quad $g\in C(\overline{\Omega}\times \R)$ and there are constants $a_1, a_2>0$ such that
 \[
 |g(x,u)| \; \leq \; a_1 + a_2|u|^s \qquad \text{with }0\leq s< \frac{N}{N-2}.
 \]
 The functional on $\Hone$ associated with (Cf'),(Cg') is $F(u,v) :=\int_\Omega f(x,u)\, v\, dx + \int_{\partial\Omega} g(x,u)\, v\, d\sigma$, and the map $u\mapsto F(u, \cdot)$ is shown in \cite{MN10} to be continuous from $\Hone$ to its dual. 
 A nice detailed proof that also argues the boundary term is given in Lemma 4.2 of \cite{MN10}.
 It is worth noting that in \cite{MN10} the nonlinearity $F$ interacts with ``eigenvalue-lines'' and here the focus is on how $F$ interacts with eigencurves.
 \end{rem}

 \begin{rem}
A simple one-dimensional case in which the condition\\
\noindent 
{\bf (C7):}  \quad {\it The map $F:\Hone \rightarrow (\Hone)^*$, as in \eqref{e.SRFv2} below, is Frech\'{e}t differentiable.}\\
holds for quite general $F$ is given in the Appendix in Section \ref{s.appendix} below; the higher dimensional case with much more general $F$ requires its own detailed careful analysis as in \cite{MN10} referenced above.
In this paper, attention is given to $f, g$ satisfying (Cf) and (Cg).
 \end{rem}

When (C1) - (C5) above hold, the bilinear forms on $\Hone$ associated with \eqref{e.SRN} are
\beq\label{e.SRabm}
\begin{cases} 
a(u,v)\; := \; \int_\Omega [ (A\nabla u)\cdot \nabla v \, + \, c\, u\, v]\, dx \, + \, \int_{\partial\Omega} b_c\, u\, v\, d\sigma, \\
  \\
 b(u, v)\; := \; \int_{\partial\Omega} b_0\, u\, v\, d\sigma,  \\
   \\
 m(u, v) \; := \; \int_\Omega m_0\, u\, v\, dx.
\end{cases} 
\eeq 
The assumptions on the data $A, b_0, b_c,c, m_0$ imply that these forms, respectively, satisfy assumptions (A1), (A2), (A3) of Section \ref{s.definitions}.  

When (Cf) and (Cg) hold, the map $F$ on $\Hone$ associated with the interior and boundary nonlinear terms of equation \eqref{e.SRN} is
\beq\label{e.SRFv2}
F(u, v)\; := \; \int_\Omega f(x,u)\, v\, dx \;  + \; \int_{\partial\Omega} g(x,u)\, v\, d\sigma \qquad \text{with } u, v\in \Hone. 
\eeq

With these bilinear and nonlinear forms defined, this section studies the following {\it nonlinear elliptic boundary-value problem}: for fixed $\lambda, \mu, \varepsilon \in \R $, find nontrivial $u\in \Hone$ satisfying
\beq\label{e.weakSRN}
a(u, v) \; = \; \lambda \, b(u, v) \, + \, \mu\, m(u, v) \, + \, \varepsilon\, F(u, v) \qquad \text{for all }v\in \Hone,
\eeq 
where $a, b, m$ are given by \eqref{e.SRabm} and $F$ is given by \eqref{e.SRFv2}.

\subsection{Linear Steklov-Robin Fredholm Alternative} 
When $\varepsilon =0$, equation \eqref{e.weakSRN} is called a {\it Steklov-Robin eigenvalue problem}.
The {\it variational eigencurves} $\lambda \mapsto \mu_k(\lambda)$ for this two-parameter bilinear eigenproblem are obtained in Section 10 of \cite{RR19}.
That section also provides regularity, orthogonality, asymptotic, and geometrical results for these eigencurves. 
Then Section \ref{s.canonicalData} above yields, for fixed $\lambda\in\R$, the {\it canonical} spectral data
\beq 
(\mu_k(\lambda)\, , \, e_k(\lambda)) \qquad \text{with }k\in \N
\eeq 
that satisfies, for all $v\in \Hone$, the following identity:
\beq\label{e.SReigenequation}
\int_\Omega [A\nabla e_k(\lambda)\cdot \nabla v + c\, e_k(\lambda)\, v\, - \mu_k(\lambda)\, m_0\, e_k(\lambda)\, v ]\, dx \; + \; \int_{\partial\Omega} [b_c\, e_k(\lambda)\, v\, - \, \lambda b_0\, e_k(\lambda)\, v ] \, d\sigma \; \; = \; \; 0.
\eeq 

The {\it linear Steklov-Robin Fredholm equation}, for fixed $(\lambda, \mu)\in \R^2$, $\ell \in (\Hone)^*$, and whose solution $u$ is sought in $\Hone$,  is given by
\beq\label{e.AppABM}
\int_\Omega [A\nabla u\cdot \nabla v + c\, u\, v\, - \mu_k(\lambda)\, m_0\, u\, v ]\, dx  +  \int_{\partial\Omega} [b_c\, u \, v\, - \, \lambda b_0\, u\, v ] \, d\sigma  \; = \; \ell(v) 
\eeq  
and is to be satisfied for all $v\in \Hone$.

The nonresonance alternative for the linear Steklov-Robin Fredholm equation is:

\btm 
Assume $(C1)$-$(C5)$, hold, and let $(\lambda, \mu) \notin \sigma(a, b, m)$ and $\ell \in (\Hone)^*$ be fixed.  
Then $\mu\neq \mu_k(\lambda)$ for all $k$, there is a unique function $\hat{u}\in \Hone$ satisfying \eqref{e.AppABM} for all $v\in \Hone$, and $\hat{u}$ has the spectral representation
\beq 
\hat{u} \; = \; \sum_{k=1}^\infty   \frac{ 1 }{\mu_k(\lambda) - \mu}  \, \ell(e_k(\lambda)) \, e_k(\lambda).
\eeq 
\etm 

Bounds for the solution $\hat{u}$ of this nonresonance problem are given by specific equivalent norms on $\Hone$ and related to Steklov-Robin eigencurves as stated next.

\btm 
Assume $(C1)$-$(C5)$, hold, and let $(\lambda, \mu)\notin \sigma(a, b,m)$ and $\ell\in (\Hone))^*$ be fixed.
Let $\tau>0$ be large enough so that the bilinear form on $\Hone$ given by
\beq a_{\lambda \tau}(u, v) \; := \; \int_\Omega [A\nabla u\cdot \nabla v + c\, u\, v\, +\; \tau \, m_0\, u\, v ]\, dx  +  \int_{\partial\Omega} [b_c\, u \, v\, - \, \lambda b_0\, u\, v ] \, d\sigma  
\eeq 
is coercive, and denote by $\|\cdot \|_{a_{\lambda \tau}^*}$ the dual norm on $(\Hone)^*$ when $\Hone$ is equipped with the norm $\| u \|_{a_{\lambda, \tau}} := \sqrt{ a_{\lambda \tau}(u, u) }$. 
	\begin{itemize}
	\item[$(i)$] When $\mu < \mu_1(\lambda)$, then the unique solution $\hat{u}$ to \eqref{e.AppABM} is bounded by 
\beq
\|\hat{u}\|_{a_{\lambda\tau}} \; \leq \; \left(   \frac{\mu_1(\lambda) + \tau}{\mu_k(\lambda) -\mu}    \right) \, \|\ell\|_{a_{\lambda\tau}^*}
\eeq
	\item[$(ii)$]  When $\mu_{k_0}(\lambda) < \mu < \mu_{k_0+1}(\lambda)$, then the unique solution $\hat{u}$ to \eqref{e.AppABM} is bounded by
\beq
\|\hat{u}\|_{a_{\lambda\tau}} \; \leq \;
 \frac{(\mu + \tau)( \mu_{k_0+1}(\lambda) - \mu_1(\lambda)}{  (\mu_1(\lambda) + \tau) (\mu_{k_0+1}(\lambda) - \mu)}\,  \| P_{k_0}\hat{u} \|_{a_{\lambda \tau}} 
\; +\; 
 \left(  \frac{\mu_{k_0+1}(\lambda) + \tau}{ \mu_{k_0+1}(\lambda) - \mu} \right) \|\ell\|_{a_{\lambda\tau}^*} 
\eeq
where 
\beq 
P_{k_0}\hat{u}  \; := \; \sum_{k=1}^{k_0} \frac{ \ell(e_k(\lambda)) }{ \mu_k(\lambda) - \mu } e_k(\lambda)
\eeq 
is the $k_0^{th}$-spectral approximation of $\hat{u}$.
	\end{itemize}
\etm

Since the bilinear forms $(a, b, m)$ for the Steklov-Robin system satisfy (C1)-(C5), and thus satisfy conditions (A1)-(A3) in \S \ref{s.definitions} above, with $V=\Hone$, these two theorems are a special case of Theorems \ref{E:NonSpec} and \ref{E:NonSpec2} of \S \ref{s.LinearFredholm}.

The resonance alternative for the linear Steklov-Robin Fredholm equation \eqref{e.AppABM} is given in the next theorem, and is a special case of Theorem \ref{E:ResFred}.
To simplify the statement, for fixed $(\lambda, \mu)\in \sigma(a,b,m)$ define
\[
J_{\lambda, \mu} \; := \; \{k\in \N :  \mu_k(\lambda)=\mu\} \qquad \text{and}\qquad E_{\lambda,\mu} \; := \; \text{span}\{e_k(\lambda) :  k\in J_{\lambda,\mu}\};
\]
this notation was used in Theorems \ref{E:ResFred} and \ref{T:Res} for the abstract results at resonance.

\btm  Assume $(C1)$-$(C5)$ hold, and let $(\lambda, \mu)\in \sigma(a, b, m)$ and  $\ell \in (\Hone)^*$ be fixed. 
Then $\mu  = \mu_{k_0}(\lambda)$ for some $k_0$ and 
\begin{itemize}
\item[$(i)$]\quad there are solutions $\hat{u}$ to \eqref{e.AppABM} if and only if $\ell(e) = 0$ for all $e\in E_{\lambda,\mu} \subset \Hone$.
\item[$(ii)$]\quad the solution set $S_{\lambda\mu}$ to \eqref{e.AppABM} is the affine subspace comprised of all $\hat{u}$ of the form 
\beq
\hat{u} \; =\; \sum_{ k\notin J_{\lambda, \mu}} \frac{1}{\mu_k(\lambda) - \mu} \, \ell(e_k(\lambda))\, e_k(\lambda) \; + \; \hat{v}
 \qquad \text{where }\quad 
\hat{v}\in E_{\lambda, \mu}
\eeq
where $e_k(\lambda) \in \Hone$ satisfy the Steklov-Robin eigenequation \eqref{e.SReigenequation} with $\mu = \mu_k(\lambda)$ and are normalized with respect to the $m$-norm, i.e. $\| e_k(\lambda) \|_m^2 : = m(  e_k(\lambda), e_k(\lambda) ) =1$.
\end{itemize}
\etm

\subsection{Nonlinear Steklov-Robin Fredholm Alternative}
Let $F:\Hone\times \Hone \rightarrow \mathbb{R}$ be the form given by  \eqref{e.SRFv2}, assume $(\lambda, \mu)\in \R^2$, and let $\varepsilon\in \R$.
The {\it nonlinear Steklov-Robin Fredholm equation}, whose solution $u$ is sought in $\Hone$,  is given by
\beq\label{e.NSRF}
\int_\Omega [A\nabla u\cdot \nabla v + c\, u\, v\, - \mu_k(\lambda)\, m_0\, u\, v ]\, dx  +  \int_{\partial\Omega} [b_c\, u \, v\, - \, \lambda b_0\, u\, v ] \, d\sigma  \; = \; \varepsilon F(u, v) 
\eeq  
and is to be satisfied for all $v\in \Hone$.

To treat the nonlinear problem \eqref{e.NSRF}, first  define the functional $\mathscr{F}:\Hone \rightarrow \mathbb{R}$ by
\beq\label{e.Functional}
\mathscr{F}(u) \; : = \; \frac{1}{p}\int_\Omega f(x, u)\, u\, dx \; + \; \frac{1}{p}\int_{\partial\Omega}g(x,u)\, u\, d\sigma 
\eeq 
The basic regularity of this functional is the following.

\btm
Assume $(Cf)$ and $(Cg)$ hold, and $\mathscr{F}$ is given by \eqref{e.Functional}.
\begin{itemize}
    \item[(i)] $\mathscr{F}$ is positive, convex, and weakly continuous on $\Hone$.
    \item[(ii)] $\mathscr{F}$ is G\^{a}teaux differentiable at $u\in \Hone$ with first variation given by
    \beq
    \delta \mathscr{F}(u;v) \; =\; |\ell_f(u)|^{p-2}\ell_f(u)\ell_f(v)  \; + \; |\ell_g(u)|^{p-2}\ell_g(u)\ell_g(v)
    \eeq 
    with direction vector $v\in \Hone$. 
    \item[(iii)] $\mathscr{F}$ has a second variation at $u$ in the directions $v, w\in \Hone$ given by
    \beq 
    \delta^2\mathscr{F}(u; v, w) \; = \; (p-1)|\ell_f(u)|^{p-2}\ell_f(v)\ell_f(w) \; + \; (p-1)|\ell_g(u)|^{p-2}\ell_g(v)\ell_g(w)
    \eeq 
\end{itemize}
\etm 

\bpf
The real-valued function $s\mapsto |s|^p$ is convex, positive, and continuously differentiable on $\mathbb{R}$.
Rewriting the functional in \eqref{e.Functional} as
\[
\mathscr{F}(u) = \frac{1}{p}|\ell_f(u)|^p + \frac{1}{p}|\ell_g(u)|^p
\]
it follows from results on compositions, and that $\ell_f$ and $\ell_g$ are continuous linear functionals, that $\mathscr{F}$ is convex, positive and continuous.
This implies $\mathscr{F}$ is weakly $l.s.c.$ on $\Hone$.
By the chain rule for G\^{a}teaux derivatives, and the formulae for the first and second (classical) derivatives of $s\mapsto  |s|^p$, the formulae for the first and second variations of $\mathscr{F}$ hold.
\epf

\bcor 
If conditions $(Cf)$, $(Cg)$ hold, then the functional $\mathscr{F}$ given by \eqref{e.Functional}  is twice Fr\'{e}chet differentiable
\ecor 

\bpf
The image of the map 
\[
u\mapsto \delta\mathscr{F}(u, \cdot)\; = \; |\ell_f(u)|^{p-2}\ell_f(u)\ell_f(\cdot )  \; + \; |\ell_g(u)|^{p-2}\ell_g(u)\ell_g(\cdot )
\]
is a sum of $\ell_f$ and $\ell_g$  in $(\Hone )^*$ with coefficients $|\ell_f(u)|^{p-2}\ell_f(u)$ and $|\ell_g(u)|^{p-2}\ell_g(u)$ being continuous in $u$ as $p>2$.
Then the image of the map
\[
u\; \mapsto \;  
\delta^2\mathscr{F}(u; \cdot , \cdot ) \; = \; 
(p-1)|\ell_f(u)|^{p-2}\ell_f(\cdot )\ell_f(\cdot ) \; + \;(p-1)|\ell_g(u)|^{p-2}\ell_g(\cdot )\ell_g(\cdot )
\]
is a sum of the continuous symmetric bilinear forms $\ell_f\otimes \ell_f$ and $\ell_g\otimes \ell_g$ with coefficients $(p-1)|\ell_f(u)|^{p-2}$ and $(p-1)|\ell_g(u)|^{p-2}$ continuous in $u$ as $p>2$.
By the basic relationship between G\^{a}teaux derivatives and Fr\'{e}chet derivatives, the desired result follows; see, for instance, Lemma 2.3.2 of Blanchard and Br\"{u}ning \cite{BlanchardBruning92}.
\epf 

With this regularity established for $\mathscr{F}$, the nonresonance alternative for the nonlinear Steklov-Robin Fredholm equation \eqref{e.NSRF} can now be stated.

\begin{thm}  
Assume $(C1)$-$(C5)$, $(Cf)$ and $(Cg)$ hold, and that the map $F$ is as in \eqref{e.SRFv2}.
If $(\lambda, \mu) \notin \sigma(a, b, m)$, then there exists a $\delta>0$ such that if $-\delta<\ve<\delta$, then \eqref{e.NSRF} has a locally unique solution $u_\ve$. 
Further, the mapping $u:(-\delta, \delta)\to H^1(\Omega)$, defined by $u(\ve)=u_\ve$, is continuously differentiable with $\ds\lim_{\ve\to 0}u(\ve)=0$.
\end{thm}

\bpf 
The map $u\mapsto F(u, \cdot)$ is Fr\'{e}chet differentiable by the previous corollary, so this theorem is a special case of Theorem \eqref{T:Nonres}.
\epf 

The (linear and nonlinear) Steklov-Robin Fredholm equation has been presented to exemplify the ease of applicability of the abstract results of this paper to quite general equations involving interior and boundary nonlinearities. Applications in the case of resonance are similar, but more delicate. For readers interested in a concrete example applying ideas similar to those in theorem \ref{T:Res} (in the case $\lam=0$) see \cite{Mar3}.


\vspace{1em}
\section{\bf Appendix: Fr\'{e}chet Differentiability of the Nonlinearity}\label{s.appendix}

In this appendix, we give an example of when the differentiability condition of Theorem \ref{T:Nonres} and Theorem \ref{T:Res} may hold.  For simplicity, the example is 1-dimsensional, but similar results hold in higher dimensions.

Let $H:=\{u \in AC[0, 1] \; : \; u' \in L^2[0, 1] \}$ be the Sobolev space with the usual norm
\[
\|u\|_{1,2} \; := \; \|u\|_2 + \|u'\|_2
\]
and suppose $F: \R \rightarrow \R$ is a given function.

\btm
If $F\in C^2(\R)$ and both $F'$ and $F''$ are bounded, then the operator $T_F: H\rightarrow H$ defined by $T_F(u)(t) \; : = \; F(u(t))$ is Fr\'{e}chet differentiable with $DT_F(u)(h)=F'(u)h$ for any $u,h \in H.$
\etm

\bpf
First we show $T_F$ is well-defined, so let $u\in H$ be fixed.
The Mean-Value Theorem together with the boundedness of $F'$ imply
\[
|F(u_2) - F(u_1)| \; \leq \; \|F'\|_{\infty}|u_2 - u_1| \qquad \text{for all }u_1, u_2\in \R.
\]
Let $\varepsilon >0$. 
Since $u\in AC[0, 1]$, choose $\delta>0$ so that 
\[
0\leq a_1\leq b_1 \leq \cdots \leq a_n\leq b_n \leq 1 \text{ with } \sum_{i=1}^n b_i - a_i \; < \; \delta \qquad \text{implies}\qquad \sum_{i=1}^n|u(b_i) - u(a_i)| \; < \; \varepsilon.
\] 
Putting these two results together gives
\[
\sum_{i=1}^n|F(u(b_i)) - F(u(a_i))| \; \; \leq \; \; \|F\|_{\infty} \sum_{i=1}^n  |u(b_i) -  u(a_i)| \; \;< \; \; \|F'\|_{\infty}\cdot \varepsilon,
\]
so that $T_F(u) \in AC[0, 1]$.
Since $T_F(u)'(t) = F'(u(t))u'(t)$ holds for almost every $t\in [0,1]$,  
\[
\|T_F(u)'\|_2 \; = \; \|F'(u)\, u' \|_2 \; \leq \; \|F'\|_{\infty}\|u'\|_2.
\]
Hence, $T_F(u) \in H$ since $u'\in L^2[0, 1]$.

Boundedness of $F'$ implies continuity of $T_F$ since for $u, v\in H$ the work above gives
\[
\|T_F(u) - T_F(v)\|_{1,2} \; \leq \; \|F'\|_\infty \left( \|u - v\|_2 + \|u' - v'\|_2\right) \; = \; \|F\|_\infty\|u - v \|_{1,2}.
\]

We finish the proof by showing that $T_F$ is continuously differentiable with derivative  $DT_F(u) = F'(u)$, that is, $DT_F(u)(h)  \; = \; F'(u)\, h$ for all $h\in H$. Note that with our notation above $DT_F(u)(h) = T_{F'}(u)h$.
As was shown above, since $F'$ is bounded, $T_{F'}$ is a map from $H$ to $H$, and since products of absolutely continuous functions are again absolutely continuous, $DT_F(u)(h)$ belongs to $AC[0, 1]$.
For almost every $t\in [0, 1]$, 
\[
(DT_F(u)(h))'(t) \; =\; F''(u(t))u'(t)h(t)  \;  +  \; F'(u(t))h'(t).
\]
The boundedness of $F', F''$ and $h$ (since $h\in AC[0,1]$) give the bound
\[
\|(DT_F(u)(h))'\|_2 \; \leq \; \|F''\|_\infty \|h\|_\infty \|u'\|_2 \; + \; \|F'\|_\infty \|h'\|_2,
\]
so that $(DT_F(u)(h))' \in L^2[0,1]$ and therefore $DT_F(u)(h) \in H$.
The estimates above give
\begin{align*}
\|DT_F(u)h\|_{1,2} \;\;  =&  \; \; \|DT_F(u)h\|_2 \;  + \;  \|(DT_F(u)h)'\|_2 \; \\
\leq &  \; \; \|F\|_\infty \|h\|_2 \; + \;  \|F''\|_\infty \|h\|_\infty \|u'\|_2 \;  +  \; \|F'\|_\infty\|h'\|_2.
\end{align*}
By the Sobolev imbedding theorem, $\|h\|_\infty \leq C\|h\|_{1,2}$ for some constant $C>0$, see Brezis \cite{B10}.
Thus, $DT_F(u)$ is a bounded linear map from $H$ to $H$.

A more careful estimate gives
\begin{align*}
\|(DT_F(u) - DT_F(v))h\|_{1,2} \; \leq  & \; \|  (F'(u) - F'(v))h\|_2 \; + \; \| (F''(u)u' - F''(v)v')h + (F'(u) - F'(v))h'\|_2 \\
\leq & \; C \|F''\|_\infty \|u - v\|_{1,2}\|h\|_{1,2},
\end{align*}
where $C$ is a new constant.
It follows that $\|DT_F(u) - DT_F(v) \|_{1,2} \; \leq \; \|F''\|_\infty \|u - v\|_{1,2}$, so the the map $u\mapsto DT_F(u)$ is continuous.

We now finish the proof by showing that $DT_F(u)(h)=F'(u)h$.  To see this note that for $u,h\in H$ we have 
\beqq
\begin{split}
|T_F(u+h)(t)-T_F(u)(t)-F'(u(t))h(t)|&=|F(u(t)+h(t))-F(u(t))-F'(u(t))h(t)|\\&=|F'(c(h,t))-F'(u(t))||h(t)|
\end{split}
\eeqq
where $c(h,t)\in[u(t),u(t)+h(t)]$.
Calculating 
\beqq
|T_F(u+h)'(t)-T_F(u)'(t)-(F'(u(t))h(t))'|
\eeqq
we get, at least for almost every $t\in[0,1]$,
\beqq
\begin{split}
|F'(u(t)+h(t))(u'(t)+h'(t))&-F'(u(t))u'(t)-F''(u(t))u'(t)h(t)-F'(u(t))h'(t)|=\\
&|F''(d(h,t))-F''(u(t))||h(t)||u'(t)|+|F'(u(t)+h(t))-F'(u(t))||h'(t)|\\
\end{split}
\eeqq
where $d(h,t)\in[x(t),x(t)+h(t)]$.
The continuity of $F', F''$ imply the differentiability of $T_F$, since by the Sobolev embedding theorem, $||h||_{1,2}\to 0$ implies $||h||_\infty\to 0$. 
\epf


\begin{thebibliography}{AMR}

\bibitem{Au05}   G. Auchmuty, ``Optimal coercivity inequalities in $W^{1, p}(\Omega)$", Proc. Roy. Soc. Edingburgh, {\bf 135 A}, (2005) pp. 915-933.

\bibitem{Au4}    G. Auchmuty, ``Bases and Comparison Results for Linear Elliptic Eigenproblems",
J.Math. Anal. Appl. \textbf{390} (2012), pp. 394-406.


\bibitem{AR16} G. Auchmuty and M.A. Rivas, ``Laplacian eigenproblems on product regions and tensor products of Sobolev spaces," J. Math. Anal. Appl. {\bf 435} (2016), pp. 842-850.

\bibitem{BlanchardBruning92} P. Blanchard and E. Br\"{u}ning,   \emph{Variational Methods in Mathematical Physics}, Springer-Verlag, Berlin (1992).

\bibitem{B10} H. Brezis, \emph{Functional Analysis, Sobolev Spaces and Partial Differential Equations}, Springer - Universitext (2010).

\bibitem{Cho20} Manki Cho, ``Steklov Expansion Method for Regularized Harmonic Boundary Value Problems," Num. Funct. Anal. Optim. {\bf 45}: 15 (2020) pp. 1871-1886.

\bibitem{DeG08} F. De Gournay, G. Allaire and F. Jouve, ``Shape and Topology Optimization of the Robust Compliance via the Level Set Method," ESAIM Control Optim. Calc. Var. {\bf 14} No. 1 (2008) pp. 43 - 70.



\bibitem{ShivSon19}  N. Fonseka, R. Shivaji, Byungjae Son and K. Spetzer, ``Classes of reaction diffusion equations where the parameter influences the equation as well as the boundary condition,"  J. Math. Anal. Appl. {\bf 476} (2019) No. 2, 480-494. 

\bibitem{ShivSon21} N. Fonseka, A. Muthunayake, R. Shivaji and Byungjae Son, ``Singular reaction diffusion equations where a parameter influences the reaction term and the boundary conditions,"  Topol. Methods Nonlinear Anal. {\bf 57} (1) 221-242, 2021.


\bibitem{ShivRob18} J. Goddard II, Q. Morris, S.B. Robinson and R. Shivaji, ``An exact bifurcation diagram for a reaction diffusion equation arising in population dynamics," Bound. Value Probl. {\bf 1}  (2018) 1-17.

\bibitem{K21} D. Kielty, ``Singular limits of sign-changing weigthed eigenproblems," Asymptotic Analysis {\bf 122}(2021)  no. 1-2,  pp. 165-200. 

\bibitem{Lewis} D. C. Lewis, ``On the role of first integrals in the perturbation of periodic solutions," Ann. Math., {\bf vol. 63}, 1956, pp.535-548.



\bibitem{MaRo1} D. Maroncelli and J. Rodr\'{i}guez, ``On the Solvability of Multipoint Boundary Value Problems for Discrete Systems at Resonance," J. Difference Equ. Appl., {\bf vol. 20}, 2013., pp 24-35.

\bibitem{Mar2}
D. Maroncelli and J. Rodr{\'i}guez,
\newblock``Weakly nonlinear boundary value problems with impulses,''
\newblock { Dyn. Contin. Discrete Impuls. Syst. Ser. A,},{\bf vol. 20}, 2013, pp. 641-656.

\bibitem{Mar3}
D. Maroncelli and  Emma Collins, ``Weakly nonlocal boundary value problems with application to Geology,'' Differ. Equ. Appl., {\bf Vol. 13}, No. 2, 2021, pp. 211-225.

\bibitem{MN10} N. Mavinga and M.N. Nkashama, ``Steklov-Neumann eigenproblems and nonlinear elliptic equations with nonlinear boundary conditions," JDE {\bf 248} (2010) pp. 1212 - 1229.

\bibitem{PS84} H.O. Peitgen and K. Schmitt, ``Global analysis of two-parameter elliptic eigenvalue problems," Trans. Amer. Math. Soc. {\bf Vol 283}, No. 1 (1984) pp. 57 -95. 



\bibitem{RR19} M.A. Rivas and S.B. Robinson, ``Eigencurves for Linear Elliptic Equations," ESAIM Control Optim. Calc. Var. {\bf 25} (2019) Art. 45, 25.






\end{thebibliography}
\end{document}